\documentclass[reqno, 11pt]{amsart}
\usepackage[utf8x]{inputenc}
\usepackage[dvips]{epsfig}
\usepackage{amsgen, amstext,amsbsy,amsopn,amssymb, amsthm, mathptmx, amsfonts,amssymb,
amscd,amsmath,nicefrac,euscript,enumerate,url,verbatim,calc,framed}
\usepackage[all]{xy}

\DeclareMathOperator{\grade}{grade}

\DeclareMathOperator{\Rees}{Rees}
\DeclareMathOperator{\reg}{reg} 
\DeclareMathOperator{\Tor}{Tor} 
\DeclareMathOperator{\tensor}{\otimes}

\newtheorem{thm}{\bf Theorem}[section]
\newtheorem{lem}[thm]{\bf Lemma}
\newtheorem{cor}[thm]{\bf Corollary}
\newtheorem{prop}[thm]{\bf Proposition}

\newtheorem{quest}[thm]{\bf Question}

\theoremstyle{definition}
\newtheorem{defn}[thm]{\bf Definition}

\newtheorem{rem}[thm]{\bf Remark}

\newtheorem{ex}[thm]{\bf Example}

\newtheorem{notation}[thm]{\bf Notation}

\title{{K}oszul property of diagonal subalgebras}

\author{Neeraj~Kumar} 
\address{Dipartimento di Matematica, Universit\'{a} di Genova,\\
Via Dodecaneso 35, 16146 Genova, Italy}
\email{kumar@dima.unige.it}

\thanks{
\endgraf
}

\begin{document}

\begin{abstract}
Let $S=K[x_1,\dots,x_n]$ be a polynomial ring over a field $K$ and $I$ a 
homogeneous ideal in $S$ generated by a regular sequence $f_1,f_2, \dots,f_k$ of 
homogeneous forms of degree $d$. We study a generalization of a result of Conca, 
Herzog, Trung, and Valla \cite{CHTV} concerning Koszul property of the diagonal 
subalgebras associated to $I$. Each such subalgebra has the form $K[(I^e)_{ed+c}]$, where 
$c,e \in \mathbb{N}$. For $k=3$, we extend \cite[Corollary $6.10$]{CHTV} 
by proving that $K[(I^e)_{ed+c}]$ is Koszul as soon as $c \geq \frac{d}{2}$ and $e >0$. 
We also extend \cite[Corollary $6.10$]{CHTV} in another direction by 
replacing the polynomial ring with a Koszul ring. 
\end{abstract}

\maketitle

\section{Introduction}
Let $S=K[x_1,\dots,x_n]$ be a polynomial ring over a field $K$ and 
$I$ a homogeneous ideal in $S$. For large $c$, the algebra $K[(I^e)_{c}]$ is 
isomorphic to the coordinate ring of some embedding of the blow up of 
$\mathbb{P}_{K}^{n-1}$ along the ideal sheaf $\tilde{I}$ in a projective space \cite{CHTV}.

Let Rees$(I)=\bigoplus_{j \geq 0}I^jt^j$ be the \emph{Rees algebra} of $I$. 
Since the polynomial ring $S[t]$ is a bigraded algebra with $S[t]_{(i,j)}=S_it^j$, 
we may consider Rees$(I)$ as a bigraded subalgebra of $S[t]$ with 
Rees$(I)_{(i,j)}=(I^j)_it^j$.

Let $c$ and $e$ be positive integers. Let 
$\bigtriangleup = \{ (cs,es) : s \in \mathbb{Z} \}$. We call 
$\bigtriangleup$ the \emph{$(c,e)$-diagonal} of $\mathbb{Z}^2$ \cite{CHTV}. 
Let $R=\oplus_{(i,j) \in \mathbb{Z}^2}R_{(i,j)}$ be a bigraded algebra, where $R_{(i,j)}$ 
denotes the $(i,j)$-th bigraded component of $R$. The $(c,e)$\emph{-diagonal subalgebra} of $R$ is 
defined as the $\mathbb{Z}$-graded algebra 
$R_{\bigtriangleup}=\bigoplus_{s \in \mathbb{Z}} R_{(cs,es)}$. Similarly, for every bigraded 
$R$-module $M$, one defines the $(c,e)$\emph{-diagonal submodule} of $M$ 
as $M_{\bigtriangleup}=\bigoplus_{s \in \mathbb{Z}} M_{(cs,es)}$. Notice that $M_{\bigtriangleup}$ 
is a module over $R_{\bigtriangleup}$. 

When $I$ is a homogeneous ideal in $S$ generated by $f_1,f_2, \dots,f_k$, of homogeneous 
forms of degree $d$, then Rees$(I)$ is a standard bigraded algebra by setting $\deg x_i=(1,0)$ and 
$\deg f_jt=(0,1)$. We observe that $K[(I^e)_{ed+c}]$ is 
the $(c,e)$-diagonal subalgebra of Rees$(I)$. We may also view $K[(I^e)_{ed+c}]$ 
as a $K$-subalgebra of $S$ generated by the forms of degree $ed+c$ in the ideal $I^e$.

\medskip

Given a field $K$, a positively graded $K$-algebra $A=\bigoplus_{i \in \mathbb{N}} A_i$ 
with $A_0=K$ is \emph{Koszul} if the field $K$, viewed as an $A$-module via the 
identification $K=A/A_{+}$, has a linear free resolution. Koszul algebras were introduced by 
Priddy \cite{Priddy} in $1970$. During the last four decades Koszul algebra have been
 studied in various contexts. Good survey on Koszul algebra is given by Fr\"oberg in 
 \cite{Froberg} during nineties and recently by Conca, De Negri and Rossi in \cite{CNR}.

Diagonal subalgebras have been studied intensively by several authors (e.g see \cite{CHTV}, \cite{STG}, 
\cite{ACN} ) because they naturally appear in Rees algebras and symmetric algebras. 
In \cite{CHTV} Conca, Herzog, Trung and Valla discuss some algebraic properties of diagonal 
subalgebras, such as Cohen-Macaulayness and Koszulness. In \cite{KSSW} Kurano et al. 
showed that Cohen-Macaulayness property holds in \cite{CHTV} 
even if the polynomial ring is replaced by a Cohen-Macaulay ring of dimension $d \geq 2$. 
In this article, we study the Koszul property of certain diagonal subalgebras 
of bigraded algebras, with applications to diagonals of Rees algebras. We generalize 
some of the important results of \cite{CHTV} regarding the Koszulness of certain 
diagonal subalgebras of bigraded algebras. 

\medskip

For any homogeneous ideal $I$, there exists integers $c_0,\;e_0$ such that the $K$-algebra 
$K[(I^e)_{ed+c}]$ is Koszul for all $c \geq c_0$ and $e \geq e_0$, see \cite[Corollary $6.9$]{CHTV}. 
If $I$ is a complete intersection ideal generated by $f_1,f_2, \dots,f_k$, of homogeneous forms of degree $d$, 
then the $K$-algebra $K[(I^e)_{ed+c}]$ is quadratic if $c \geq \frac{d}{2}$ and $e >0$; furthermore 
$K[(I^e)_{ed+c}]$ is Koszul if $c \geq \frac{d(k-1)}{k}$ and $e>0$, see \cite[Corollary $6.10$]{CHTV}. 

\medskip

The main results of this paper are the following:
\begin{itemize}
\item[(i)] Let $I$ be an ideal of the polynomial ring $K[x_1,\dots,x_n]$ generated by a regular sequence 
 $f_1,f_2, f_3$, of homogeneous forms of degree $d$. Then $K[(I^e)_{ed+c}]$ is 
 Koszul for all $c \geq \frac{d}{2}$ and $e >0$. 
 \item[(ii)] Let $A$ be a standard graded Koszul ring. Let $I$ be an ideal of $A$ generated by a 
regular sequence $f_1,f_2, \dots,f_k$, of homogeneous forms of degree $d$. Then $K[(I^e)_{ed+c}]$ is 
 Koszul for all $c \geq \frac{d(k-1)}{k}$ and $e >0$.
\end{itemize}

\medskip

Let $R$ be a standard bigraded $K$-algebra. In Section $2$, we study homological 
properties of shifted modules $R(-a,-b)_{\bigtriangleup}$, 
which play an important role in the transfer of homological information from 
$R$ to $R_{\bigtriangleup}$. It is important to bound 
the homological invariants of the shifted diagonal module $R(-a,-b)_{\bigtriangleup}$ as an 
$R_{\bigtriangleup}$-module. For a bigraded polynomial ring $R$, it is proved in \cite{CHTV} that
$R(-a,-b)_{\bigtriangleup}$ has a linear $R_{\bigtriangleup}$ resolution. 
Proposition \ref{main-prop} is an extension of \cite[Theorem $6.2$]{CHTV} for 
certain bigraded complete intersection ideal and crucial in proving Theorem \ref{main-thm}.

\medskip

Let $S=K[x_1,\dots,x_n]$ be a polynomial ring. In \cite[p.900]{CHTV} the authors mentioned 
that for a complete intersection ideal $I$ in $S$ generated by $f_1,f_2, \dots,f_k$, of 
homogeneous forms of degree $d$, the algebra $K[(I^e)_{ed+c}]$ is expected to be Koszul as soon 
as $c \geq \frac{d}{2}$. For $k=1,2$, it is obvious. The first nontrivial case is $k=3$. 
In Section $3$, we answer their expectation affirmatively for $k=3$, see Theorem \ref{main-thm}. 
The motivation for such generalization came from the work of Caviglia \cite{GC}, and 
Caviglia and Conca \cite{GC-AC}. Note that for $k=3$, the result of \cite{GC-AC} is just the case: 
$d=2$ and $c=1$; furthermore the main result of \cite{GC} correspond to the case: $d=2,\;c=1$, 
$f_1=x_1^2,f_2=x_2^2,f_3=x_3^2$ and $n=3$.

\medskip

In Section $4$, we generalize \cite[Theorem $6.2$]{CHTV} and some of its relevant corollaries. 
The main result of this section is Theorem \ref{poly to koszul change thm}, which is a generalization 
of \cite[Corollary $6.10$]{CHTV} regarding the koszulness of certain diagonals of the Rees algebra of an 
ideal in the polynomial ring. We show that the Koszulness property holds even if the polynomial 
ring is replaced by a Koszul ring. The reason for such generalization comes from the fact 
that from certain point of view, Koszul algebras behave homologically as polynomial rings.

\medskip

\paragraph{{\bf Acknowledgment:}}The author is very grateful to Professor Aldo Conca for several useful 
discussions on the subject of this article. The author thanks Professor Giulio Caviglia 
for his helpful suggestions concerning the presentation of this article. The author is also 
grateful to the referee for a number of helpful suggestions for improvement in the article.
\medskip
\section{Generalities and preliminary results} 
\medskip

Let $A$ be a standard graded $K$-algebra i.e. $A =\oplus A_i={S}/{I}$, where $S$ is a 
polynomial ring and $I$ a homogeneous ideal of $S$.
For a finitely generated graded $A$-module $M=\oplus M_i$, set
\[
 t_i^{A}(M)=\sup \{j : \Tor_{i}^{A}(M,K)_j \neq 0 \},
\]
with $t_i^{A}(M)=-\infty$ if $\Tor_i^A(M,K) = 0.$
\begin{defn}The \emph{Castelnuovo-Mumford regularity} $\reg_A(M)$ of
an $A$-module $M$ is defined to be
\[
 \reg_AM= \sup \{t_i^{A}(M)-i: i \geq 0\}.
\]
When $A=S$ is a polynomial ring, one has
\[
 \reg_SM=\max\{t_i^{S}(M)-i: i \geq 0\}.
\]
For a polynomial ring $S=K[x_1,\dots,x_n]$, one can also compute $\reg_SM$ via the local cohomology modules 
$H_{\mathfrak{m}}^{i}(M)$ for $i=0,1,\dots,n$. One has
\[
\reg_SM=\max\{j+i:H_{\mathfrak{m}}^{i}(M)_j \neq 0\}. 
\]
\end{defn}

\begin{defn}
{\bf{Koszul algebra:}} A standard graded $K$-algebra $A$ is said to be a \emph{Koszul algebra}
if the residue field $K$ has a linear $A$-resolution. Equivalently, $A$ is Koszul when $ \reg_A(K)=0$.
\end{defn}
\begin{ex}
Let $A={K[x]}/{(x^2)}$, then $K$ has a linear $A$-resolution
\[
\cdots \rightarrow A(-2) \xrightarrow{\bar{x}}  A(-1) \xrightarrow{\bar{x}} A \rightarrow K \rightarrow 0.
\]
\end{ex}
The property of being a Koszul algebra is preserved under various constructions, in particular 
under taking tensor products, Segre products and Veronese subrings, see Backelin and Fr\"oberg \cite{BF}.

\medskip

Let $A$ be a Koszul algebra, and $S$ be the polynomial ring mapping onto $A$. Then the regularity of 
any finitely generated graded module $M$ over $A$ is always finite; in fact, $\reg_AM \leq \reg_SM $, 
see Avramov and Eisenbud \cite[Theorem $1$]{AE}. If $M=\oplus_{i=a}^{b}M_i$ with $M_b\neq 0$, then 
\begin{eqnarray}\label{koszul-test-2}
\reg_AM \leq \reg_SM=b.
\end{eqnarray}
If $A=\oplus_{i\geq 0} A_i$ is a graded algebra, then the $c$-th \emph{Veronese subalgebra} 
is $A^{(c)}=\oplus_{i \geq 0} A_{ic}$. An element in $A_{ic}$ is considered to have degree $i$. 
\begin{defn}
Consider a standard graded $K$-algebra $A$. Given $k,m \in \mathbb{N}, $ and $0\leq k<m$, we set
\[
V_A(m,k)=\bigoplus_{i \in \mathbb{N}} A_{im+k}.
\]
We observe that $A^{(m)}=V_A(m,0)$ is the usual $m$-th {Veronese
subring} of $A$, and that the $V_A(m,k)$ are $A^{(m)}$-modules known
as the \emph{Veronese modules} of $A$. For a finitely generated
graded $A$-module $M$, similarly we define
\[
M^{(m)}=\bigoplus_{i \in \mathbb{Z}} M_{im}.
\]
We consider $A^{(m)}$ as a standard graded $K$-algebra with
homogeneous component of degree $i$ equal to $A_{im}$, and $M^{(m)}$
as a graded $A^{(m)}$-module with homogeneous components $M_{im}$ of degree $i$.
\end{defn}

\begin{defn}
Let $A$ and $B$ be positively graded $K$-algebra. Denote by $A \underline{\otimes} B $ 
the \emph{Segre product}
\[
 A \underline{\otimes} B= \bigoplus_{i \in \mathbb{N}} A_i \otimes_K B_i,
\]
of $A$ and $B$. Given graded modules $M$ and $N$ over $A$ and $B$, 
one may form the Segre product
\[
M \underline{\otimes} N= \bigoplus_{i \in \mathbb{Z}} M_i \otimes_K N_i ,
\]
of $M$ and $N$. Clearly $M \underline{\otimes} N$ is a graded $A \underline{\otimes} B$- module.
\end{defn}
A beautiful introduction to construction of multigraded objects including Segre products is given by 
Goto and Watanabe in \cite[Chapter $4$]{GW}. Segre products have been studied in the sense of 
Koszulness by several authors e.g. Backelin et al. \cite{BF}, Eisenbud et al. \cite{ERT}, 
Conca et al. \cite{CHTV}, Fr\"oberg \cite{Froberg} and Blum \cite{Blum}. We will use their results 
at several occasions in this paper. 

\medskip

Let $A$ and $B$ be Koszul $K$-algebras. Let $M$ be a finitely generated 
graded $A$-module and $N$ be a finitely generated graded $B$-module. 
Assume $M$ and $N$ have linear resolutions over $A$ and $B$
respectively. Also assume $M \underline{\otimes} N \neq 0$. 
Then by \cite[Lemma $6.5$]{CHTV}, $M \underline{\otimes} N$ has a 
linear $A \underline{\otimes} B$-resolution and  
\begin{eqnarray}\label{regularity relation segre product}
 \reg_{A \underline{\otimes} B}{M \underline{\otimes} N}=\max \{ \reg_AM,\; \reg_BN\}.
\end{eqnarray}
We will use this relation on regularity for Segre products in the 
proof of Theorem \ref{koszul-thm-gen} and Theorem \ref{poly to koszul change thm}.

\begin{rem}
It is to be noticed that \cite[Lemma $6.5$]{CHTV} also needs the hypothesis $M \underline{\otimes} N \neq 0$.
 For instance, Assume $M=K$ and $N=K(-1)$, then $M \underline{\otimes} N =0$ and this leads 
 to inconsistency in $(\ref{regularity relation segre product})$.  
\end{rem}

The following lemma is very useful. We will use this lemma in the proof of 
Proposition \ref{main-prop}, Theorem \ref{main-thm} and Theorem \ref{poly to koszul change thm}.

\begin{lem}\label{regularity-complex-homology}{\rm (Technical Lemma)} \\[.1mm]
 Let
\[
{\bf{M}}:\cdots \to M_{k+1}\to M_{k}\to M_{k-1}\to \cdots \to M_1 \to M_0\to  0,
\]
be a complex of graded $A$-modules with maps of degree $0$. Set $H_i=H_i({\bf{M}})$. 
Then for every $i \geq 0$ one has
\begin{eqnarray}\label{1st eqn-regularity-complex-homology}
t^A_i(H_0)\leq \max\{ \alpha,\beta \}, 
\end{eqnarray}
where $\alpha= \sup \{t^A_{i-j}(M_j) : j=0,\dots ,i\}$ and $\beta= \sup \{t^A_{i-j-1}(H_j) : j=1,\dots ,i-1\}$. 
Moreover one has 
\begin{eqnarray}\label{2nd eqn-regularity-complex-homology}
\reg_A(H_0) \leq \max\{ \alpha^{\prime},\beta^{\prime} \}, 
\end{eqnarray}
where $\alpha^{\prime}= \sup \{\reg_A(M_j)-j : j\geq 0 \}$ and $\beta^{\prime}= \sup \{\reg_A(H_j) -(j+1): j \geq 1\}$.
\end{lem}
\begin{proof}
Let $Z_i,\;B_i,\;H_i$ denotes the $i$-th cycles, boundaries and homology modules 
respectively. We have short exact sequences
\[
 0 \rightarrow B_i \rightarrow Z_i \rightarrow H_i \rightarrow 0,
\]
and
\[
 0 \rightarrow Z_{i+1} \rightarrow M_{i+1} \rightarrow B_i \rightarrow 0,
\]
with $Z_0=M_0$. 
Therefore, by \cite[Lemma $2.2(a)$]{BCR}, one has
\[
 \begin{split}
  t^A_i(H_0)&\leq \max\{ t^A_i(M_0) , t^A_{i-1}(B_0)\},\\
  t^A_{i-1}(B_0)&\leq \max\{ t^A_{i-1}(M_1) , t^A_{i-2}(Z_1)\},\\
  t^A_{i-2}(Z_1)&\leq \max\{ t^A_{i-2}(B_1) , t^A_{i-2}(H_1)\},\\
  t^A_{i-2}(B_1)&\leq \max\{ t^A_{i-2}(M_2) , t^A_{i-3}(Z_2)\}, \text{ and so on.}
 \end{split}
\]
Summarizing the details, one has: 
\begin{eqnarray*}
 t^A_i(H_0) \leq \max\{ t^A_i(M_0),\;t^A_{i-1}(M_1)\;,\;\dots,\;t^A_{i-j}(M_j),\; 
 t^A_{i-2}(H_1),\; \dots,\; t^A_{i-j-1}(H_j) \}.
\end{eqnarray*}
Take $\alpha= \sup \{t^A_{i-j}(M_j) : j=0,\dots ,i\}$ and 
$\beta= \sup \{t^A_{i-j-1}(H_j) : j=1,\dots ,i-1\}$, then we obtain the desired result 
$(\ref{1st eqn-regularity-complex-homology})$ for $t^A_i(H_0)$. The second inequality 
$(\ref{2nd eqn-regularity-complex-homology})$ follows from 
$( \ref{1st eqn-regularity-complex-homology})$.
\end{proof}

Let $R=\oplus_{(i,j) \in \mathbb{Z}^2}R_{(i,j)}$ be a bigraded standard $K$-algebra. 
Here standard means that $R_{(0,0)}=K$ and $R$ is generated as a $K$-algebra by 
the $K$-vector spaces $R_{(1,0)}$ and $R_{(0,1)}$ of finite dimension.
\begin{defn}
{\bf{Diagonal subalgebra:}} 
Let $R$ be a bigraded standard $K$-algebra. Let $c$ and $e$ be positive integers. 
Let $\bigtriangleup$ be the $(c,e)$-diagonal of $\mathbb{Z}^2$. The $(c,e)$-\emph{diagonal 
subalgebra} $R_{\bigtriangleup}$ of $R$ is defined as 
\[
 R_{\bigtriangleup}=\bigoplus_{s \in \mathbb{Z}} R_{(cs,es)}.
\]
\end{defn}
We observe that $R_{\bigtriangleup}$ is the $K$-subalgebra of $R$ generated by $R_{(c,e)}$ 
and hence it is a standard graded $K$-algebra. Similarly, one defines the $(c,e)$-diagonal 
submodule of any bigraded $R$-module $M=\oplus_{(i,j) \in \mathbb{Z}^2}M_{(i,j)}$ as 
$M_{\bigtriangleup}=\bigoplus_{s \in \mathbb{Z}} M_{(cs,es)}$. Notice that $M_{\bigtriangleup}$ 
is a module over $R_{\bigtriangleup}$. The map $M \longmapsto M_{\bigtriangleup}$, being a 
selection of homogeneous components, defines an exact functor from the category of bigraded 
$R$-modules and maps of degree $0$ to the category of graded $R_{\bigtriangleup}$-modules 
with maps of degree $0$. 

\begin{notation}
We have $\bigtriangleup = \{ (cs,es) : s \in \mathbb{Z} \}$, but the bounds on $c$ and $e$ change 
from time to time. Note that $c,s,\;\bigtriangleup$ will always be used in this way, with $c$ and $s$ changing 
as described. For a real number $\alpha$, we use $\lceil \alpha \rceil$ for the 
smallest integer $m$ such that $m \geq \alpha$.
\end{notation}
For $(a,b) \in \mathbb{Z}^2$, let $R(-a,-b)$ be a shifted copy of $R$. By definition
\[
 R(-a,-b)_{\bigtriangleup} = \bigoplus_{s \in \mathbb{Z}} R_{(-a+cs,-b+es)}.
\]
Since $R$ is positively graded, we may consider only those $s \in \mathbb{Z}$ 
for which $-a+cs \geq 0$ and $-b+es \geq 0$. 
Assume $\max {( \lceil \frac{a}{c} \rceil,\lceil \frac{b}{e} \rceil )}=\lceil \frac{a}{c} \rceil$. Then
\[
 R(-a,-b)_{\bigtriangleup} = \bigoplus_{s \geq \lceil \frac{a}{c} \rceil} R_{(-a+cs,-b+es)}.
\]
Therefore $R(-a,-b)_{\bigtriangleup}$ is a $R_{\bigtriangleup}$-submodule of $R$ generated by 
$R_{(-a+c \lceil \frac{a}{c} \rceil ,-b+ e\lceil \frac{a}{c} \rceil)}$. 
The other case is similar, summarizing the details, one has: 
\[
R(-a,-b)_{\bigtriangleup} = 
\begin{cases}
  R(-a+c \lceil \frac{a}{c} \rceil,-b+ e\lceil \frac{a}{c} \rceil)_{\bigtriangleup}(- \lceil  \frac{a}{c} \rceil),  
  &\text{ if  $ \lceil \frac{a}{c} \rceil \geq  \lceil \frac{b}{e} \rceil $;}\\
  R(-a+c \lceil \frac{b}{e} \rceil,-b+ e\lceil \frac{b}{e} \rceil)_{\bigtriangleup}(- \lceil  \frac{b}{e} \rceil),  
  &\text{ if  $ \lceil \frac{a}{c} \rceil \leq \lceil \frac{b}{e} \rceil $.} 
 \end{cases} 
\]
The homological properties of the shifted diagonal module $R(-a,-b)_{\bigtriangleup}$ play an important role 
in the transfer of homological information from $R$ to $R_{\bigtriangleup}$. In the following proposition, 
we try to bound the homological invariants (regularity) of the shifted diagonal module $R(-a,-b)_{\bigtriangleup}$ as an 
$R_{\bigtriangleup}$-module. The following proposition is crucial in proving Theorem \ref{main-thm}.

\begin{prop}\label{main-prop}
Let $S=K[x_1,\dots,x_m,t_1,\dots,t_n]$ be a polynomial ring bigraded by $\deg x_i=(1,0)$ 
for $i=1,\dots,m$ and $\deg t_i=(0,1)$ for $i=1,\dots,n$. 
Let $I$ be an ideal of $S$ generated by a regular sequence with elements all of bidegree $(d,1)$ 
and $R={S}/{I}$. Let $\frac{d}{2} \leq c < \frac{2d}{3}$ and $e >0$. Then:
\begin{enumerate}
 \item $R_{\bigtriangleup}$ is Koszul.
 \item $\reg_{R_{\bigtriangleup}} {R(-a,-b)_{\bigtriangleup}} \leq
\max \{ \lceil \frac{a}{c} \rceil,\lceil \frac{b}{e} \rceil  \}.$
\end{enumerate}
\end{prop}

\begin{proof}
Let $h$ be the codimension of $I$. The proof is by induction on $h$. 
If $h=0$, then $R_{\bigtriangleup}$ is the Segre product of $K[x_1,\dots,x_m]^{(c)}$ 
and $K[t_1,\dots,t_n]^{(e)}$. Thus $R_{\bigtriangleup}$ is Koszul by \cite{BF}. 
For $(b)$, see \cite[proof of Theorem $6.2$]{CHTV}.

Assume $h>0$. We may write $R={T}/{(f)}$ where $f$ is a $T$-regular 
element of bidegree $(d,1)$ and where $T$ is defined as the quotient of $S$ by an $S$-regular 
sequence of length $h-1$ of elements of bidegree $(d,1)$. We have a short 
exact sequence of $T$-modules:
\begin{eqnarray}\label{2nd exact seq}
 0 \longrightarrow T(-d,-1) \longrightarrow T \longrightarrow R \longrightarrow 0
\end{eqnarray} 
and applying $-_{\bigtriangleup}$, we have an exact sequence of $T_{\bigtriangleup}$-modules:
\[
 0 \longrightarrow T(-d,-1)_{\bigtriangleup} \longrightarrow T_{\bigtriangleup} 
\longrightarrow R_{\bigtriangleup} \longrightarrow 0.
\]
By induction we know $T_{\bigtriangleup}$ is Koszul and that 
$ \reg_{T_{\bigtriangleup}}{T(-d,-1)_{\bigtriangleup}} \leq \lceil \frac{d}{c} \rceil$. 
As $ \frac{d}{c} \leq 2$, one has
\[
 \reg_{T_{\bigtriangleup}}{R_{\bigtriangleup}} \leq 1.
\]
By \cite[Lemma $2.1(3)$]{GC-AC}, we may conclude that $R_{\bigtriangleup}$ is Koszul, as 
$T_{\bigtriangleup}$ is Koszul by induction.
Now to prove $(b)$ we consider the following two cases.\\ 
{\bf{Case $1$.}} Assume $\lceil \frac{a}{c} \rceil < \lceil \frac{b}{e} \rceil$. 
Shift $(\ref{2nd exact seq})$ by $(-a,-b)$ and then apply $-_{\bigtriangleup}$, we get a short exact sequence 
of $T_{\bigtriangleup}$-modules:
\[
 0 \longrightarrow T(-a-d,-b-1)_{\bigtriangleup} \longrightarrow T(-a,-b)_{\bigtriangleup}
 \longrightarrow R(-a,-b)_{\bigtriangleup} \longrightarrow 0.
\] 
So we have: 
\[
 \reg_{T_{\bigtriangleup}}{R(-a,-b)_{\bigtriangleup}} \leq \max 
  \{ \; \reg_{T_{\bigtriangleup}}{T(-a,-b)_{\bigtriangleup}},\; 
\reg_{T_{\bigtriangleup}}{T(-a-d,-b-1)_{\bigtriangleup}} -1\; \}.
\]
By induction, one has $ \reg_{T_{\bigtriangleup}}{T(-a,-b)_{\bigtriangleup}} \leq \lceil \frac{b}{e} \rceil$, and
\[
\reg_{T_{\bigtriangleup}}{T(-a-d,-b-1)_{\bigtriangleup}} \leq
\max \{ \lceil \frac{a+d}{c} \rceil, \lceil \frac{b+1}{e} \rceil \}.
\]
Since 
\[
 \lceil \frac{a+d}{c} \rceil \leq \lceil \frac{a}{c} \rceil + \lceil \frac{d}{c} \rceil \leq (\lceil \frac{b}{e} \rceil -1) +2 
 = \lceil \frac{b}{e} \rceil +1 \text{ and }\lceil \frac{b+1}{e} \rceil \leq \lceil \frac{b}{e} \rceil + 1,
\]
we conclude that $\reg_{T_{\bigtriangleup}}{T(-a-d,-b-1)_{\bigtriangleup}} \leq \lceil \frac{b}{e} \rceil + 1$. 
Thus we have 
\begin{eqnarray}\label{some-equality}
\reg_{T_{\bigtriangleup}}{R(-a,-b)_{\bigtriangleup}} \leq \lceil \frac{b}{e} \rceil.
\end{eqnarray}
Since we have already shown that $ \reg_{T_{\bigtriangleup}}R_{\bigtriangleup} \leq 1$, 
we may conclude by \cite[Lemma $2.1(1)$]{GC-AC} that 
\[
 \reg_{R_{\bigtriangleup}}{R(-a,-b)_{\bigtriangleup}} \leq \lceil \frac{b}{e} \rceil.
\]
{\bf{Case $2$.}} Assume $\lceil \frac{a}{c} \rceil \geq \lceil \frac{b}{e} \rceil$. 
Set $P=(t_1,\dots,t_n) \subset S.$ We have
\begin{eqnarray*}
 R(-a,-b)_{\bigtriangleup}= R(-a+c \lceil \frac{a}{c} \rceil,-b+ e\lceil \frac{a}{c} \rceil)_{\bigtriangleup}
(- \lceil  \frac{a}{c} \rceil).
\end{eqnarray*}
So we have to prove that $\reg_{R_{\bigtriangleup}}{R(\alpha,\beta)_{\bigtriangleup}} \leq 0$ 
where $\alpha=-a+c \lceil \frac{a}{c} \rceil $ and $\beta=-b+ e\lceil \frac{a}{c} \rceil$. 
Consider the minimal free (bigraded) resolution of $S/{P^{\beta}}$ as an $S$-module:
\begin{eqnarray*}
 {\bf{F}}:\quad 0 \longrightarrow F_n \longrightarrow F_{n-1} \longrightarrow \cdots 
\longrightarrow F_1 \longrightarrow F_0 \longrightarrow 0, 
\end{eqnarray*}
with $F_0=S$ and $F_i=S^{\sharp}(0,-\beta -i+1)$ for $i >0$ where $\sharp$ denotes 
some integer depending on $n,\beta$ and $i$ that is irrelevant in our discussion. 
The homology of ${\bf{F}} \tensor R$ is $\Tor_{\bullet}^{S}(S/{P^{\beta}},R)$. 
We may as well compute $\Tor_{\bullet}^{S}(S/{P^{\beta}},R)$ as the homology of 
$S/{P^{\beta}} \tensor {\bf{G}}$ where ${\bf{G}}$ is a free resolution 
of $R$ as an $S$-module. By assumption, we may take ${\bf{G}}$ to be a Koszul 
complex on a sequence of elements of bidegree $(d,1)$. It follows that:
\begin{eqnarray*}
 H_i({\bf{F}} \tensor R) = 
\begin{cases}
   \text{a subquotient of } (S/{P^{\beta}})^{\sharp}(-di,-i),  &\text{ if  $ 0 \leq i \leq h $;}\\
    0,  &\text{ if $ i> h$.}
\end{cases}
\end{eqnarray*}
Shifting with $(\alpha,\beta)$ and applying $-_{\bigtriangleup}$ we have a complex 
$({\bf{F}} \tensor R(\alpha,\beta))_{\bigtriangleup}$. We claim this complex 
has no homology. Shifting and applying $-_{\bigtriangleup}$ are compatible operations 
with taking homology. Therefore to prove $({\bf{F}} \tensor R(\alpha,\beta))_{\bigtriangleup}$ 
has no homology at all, we only need to check that 
\begin{eqnarray}\label{inequality case1}
 [(S/{P^{\beta}})(-di+\alpha,-i+\beta)]_{\bigtriangleup}=0 \;\; \text{ for all $i$.}
\end{eqnarray}
To prove $(\ref{inequality case1})$, take the $j$-th degree component, 
\[
[[(S/{P^{\beta}})(-di+\alpha,-i+\beta)]_{\bigtriangleup} ]_j= (S/{P^{\beta}})_{(cj-di+\alpha,ej-i+\beta)}.
\]
We will show that 
\begin{eqnarray}\label{ext-inq}
 (S/{P^{\beta}})_{(cj-di+\alpha,ej-i+\beta)}=0. 
\end{eqnarray}
The case $i=0$ will be dealt separately. Assume $i >0$. Clearly $(\ref{ext-inq})$ holds 
if $ej-i+\beta \geq \beta$, that is, if $ej \geq i$. 
 To complete the argument for $(\ref{ext-inq})$, it is enough to show that $cj-di+\alpha <0 $ for $ej <i$. 
 Let $a=qc+r\; ; \; 0 \leq r <c$, then 
\[
 \lceil \frac{a}{c} \rceil = \begin{cases}
    q+1,  &\text{ if  $r \neq 0 $;}\\
    q,     &\text{ if $r=0$.}
    \end{cases} 
\]
Therefore
\begin{eqnarray}\label{alpha value}
cj-di +\alpha =  \begin{cases}
    cj-di+c-r,  &\text{ if  $r \neq 0 $;}\\
    cj-di,     &\text{ if $r=0$.}
    \end{cases} 
\end{eqnarray}
Assume $r=0$ and $ej<i$. It is easy to see that $$cj-di < i(\frac{c}{e}-d)< i(\frac{2d}{3e}-d)<0.$$
Assume $r \neq 0$ and $ej<i$. By assumption we have $j \leq \frac{i}{e} - \frac{1}{e}$ and $cj-di+c-r < c(j+1)-id$. 
One has 
\begin{eqnarray}\label{ext-eqn1}
 c(j+1)-di \leq c(\frac{i}{e}-\frac{1}{e}+1)-di=i(\frac{c}{e}-d)-\frac{c}{e}+c.
\end{eqnarray}
As $\frac{d}{2} \leq c < \frac{2d}{3}$ and $e>0$, we may write 
\begin{eqnarray}\label{ext-eqn2}
  \frac{c}{e}-d<\frac{2d}{3e}-d= \frac{d(2-3e)}{3e} \text{ and } -\frac{c}{e} \leq - \frac{d}{2e}.
\end{eqnarray}
Thus by $(\ref{ext-eqn1})$ and $(\ref{ext-eqn2})$, we have
\[
 c(j+1)-di < i(\frac{d(2-3e)}{3e})-\frac{d}{2e}+\frac{2d}{3}=\frac{d}{6e} [i(4-6e)+(4e-3)]. 
\]
It is easy to see that $i(4-6e)+(4e-3)<0$ for all $i>0$. Assume $i=0$. Denote by ${\bf{C}}$, the complex 
$({\bf{F}} \tensor R(\alpha,\beta))_{\bigtriangleup}$. If $H_0({\bf{C}}) \neq 0$, then either $ej + \beta < \beta$ or 
$cj + \alpha \geq 0$ in $(\ref{ext-inq})$. Thus $H_0({\bf{C}}) \neq 0$ if $ \frac{-\alpha}{c} \leq j < 0$, 
which is not possible, since $j$ has to be an integer and by previous discussion in $(\ref{alpha value})$, 
one has $ -1 < \frac{-\alpha}{c} \leq 0$. Since $H_i({\bf{C}}) = 0$ for all $ i \geq 0$, we have the 
following exact complex ${\bf{C}}$:
\[ 
0 \longrightarrow R(\alpha, -i+1)_{\bigtriangleup}  \longrightarrow \cdots 
\longrightarrow R(\alpha,-1)_{\bigtriangleup}
\longrightarrow  R(\alpha,0)_{\bigtriangleup}    \longrightarrow    R(\alpha,\beta)_{\bigtriangleup} \longrightarrow 0.  
\]
From the exact complex ${\bf{C}}$, we build another complex:
\[ 
{\bf{T}}:\;\;    0 \longrightarrow R(\alpha, -i+1)_{\bigtriangleup}  \longrightarrow \cdots 
\longrightarrow R(\alpha,-1)_{\bigtriangleup}
\longrightarrow  R(\alpha,0)_{\bigtriangleup}    \longrightarrow  0.  
\]
Then the homology of the new complex ${\bf{T}}$ is given by
\[
 H_i({\bf{T}}) = 
\begin{cases}
  R(\alpha,\beta)_{\bigtriangleup} ,  &\text{ if $i=0$;}\\
    0,  &\text{ if $ i> 0$.}
\end{cases}
\]
Thus Lemma \ref{regularity-complex-homology} applied to the complex ${\bf{T}}$, one has
\[
 \reg_{R_{\bigtriangleup}}{R(\alpha,\beta)_{\bigtriangleup}} \leq 
 \max \{  \reg_{R_{\bigtriangleup}}{R(\alpha,-i)_{\bigtriangleup}}  -i: i \geq 0  \}.
\]
Note that by Case $1$, we have 
$\reg_{R_{\bigtriangleup}}{R(\alpha,-i)_{\bigtriangleup}} \leq \lceil \frac{i}{e} \rceil $, since 
$\lceil \frac{- \alpha }{c} \rceil \leq \lceil \frac{i}{e} \rceil$. Thus we conclude that 
$ \reg_{R_{\bigtriangleup}}{R(\alpha,\beta)_{\bigtriangleup}} \leq 0$. Hence the claim $(b)$ follows.
\end{proof}

\begin{rem}
Note that the Proposition \ref{main-prop} is an extension of \cite[Theorem $6.2$]{CHTV} for 
certain bigraded complete intersection ideal. Note also that the statement of 
Proposition \ref{main-prop} is similar to (and more general then) \cite[Proposition $2.2$]{GC-AC}.
\end{rem}

\section{Improvement of bounds}
\medskip

Let $S=K[x_1,\dots,x_n]$ be a polynomial ring. Let $I$ be a homogeneous 
ideal in $S$ generated by a regular sequence $f_1,f_2, \dots,f_k$, of homogeneous forms 
of degrees $d$. Let $c$ and $e$ be positive integers. Consider the $K$-subalgebra 
of $S$ generated by the homogeneous forms of degree $ed+c$ in the ideal $I^e$, that is, $K[(I^e)_{ed+c}]$. 
We have seen that $K[(I^e)_{ed+c}]$ is the $(c,e)$-diagonal subalgebra of Rees$(I)$. 

Recall that the $K$-algebra $K[(I^e)_{ed+c}]$ is shown to be quadratic if $c \geq \frac{d}{2}$ 
and Koszul if $c \geq \frac{d(k-1)}{k}$, see \cite[Corollary $6.10$]{CHTV}. In \cite[p.900]{CHTV} the 
authors mentioned that, they expect $K[(I^e)_{ed+c}]$ to be Koszul also for $\frac{d}{2} \leq c < \frac{d(k-1)}{k}$. 
For $k=1,2$, it is obvious that $K[(I^e)_{ed+c}]$ is Koszul. The very first nontrivial 
instance of this problem occurs for $d=2$ and $k=3$, in which case, only possible value is $c=1$. 
For $c=1$, the answer is positive and solved by Caviglia and Conca in \cite{GC-AC}.  

In this Section, for $k=3$ and for any $d$, we prove that $K[(I^e)_{ed+c}]$ 
is Koszul as soon as $c \geq \frac{d}{2}$ and $e >0$. The main theorem of this section is as follows:

\begin{thm}\label{main-thm} Let $I$ be an ideal of the polynomial ring $K[x_1,\dots,x_n]$ 
generated by a regular sequence $f_1,f_2, f_3$, of homogeneous forms of degree $d$. 
Then $K[(I^e)_{ed+c}]$ is Koszul for all $c \geq \frac{d}{2}$ and $e >0$. 
\end{thm}

We consider the Rees algebra, Rees$(I) \subset S[t]$ of $I$ with its standard bigraded 
structure induced by deg$(x_i)=(1,0)$ and deg$(f_jt)=(0,1)$. It can be realized as a 
quotient of the polynomial ring $S^{\prime}=K[x_1,\dots,x_n,t_1,t_2,t_3]$ 
bigraded with deg$(x_i)=(1,0)$ and deg$(t_j)=(0,1)$, by the ideal $J$ 
generated by the $2$-minors of
\[
 M=\begin{pmatrix}
    f_1 & f_2 & f_3 \\
    t_1 & t_2 & t_3    
   \end{pmatrix}.
\]
Let $h_1,h_2,h_3$ be the $2$-minors of $M$ with the appropriate sign, say $h_i$ equal 
to $(-1)^{i+1}$ times the minor of $M$ obtained by deleting the $i$-th column. Hence
\[
 J=I_2(M)=(h_1,h_2,h_3).
\]
The sign convention is chosen so that the rows of the matrix $M$ are syzygies of $h_1,h_2,h_3$. We 
will use the following lemma to prove Theorem $\ref{main-thm}$.
\begin{lem}\label{homology computation} {\rm (Technical Lemma)} \\[1mm]
$(i)$ $h_1,h_2$ form a regular $S^{\prime}$-sequence.\\
$(ii)$  $(h_1,h_2):h_3=(f_3,t_3)$.\\
$(iii)$ $(h_1,h_2):t_3=J$.\\
$(iv)$ $(t_3,h_1,h_2):f_3=(t_1,t_2,t_3)$.
\end{lem}
\begin{rem}
Note that Lemma \ref{homology computation} is proved for $d=2$ in \cite[Lemma $3.1$]{GC-AC}. 
We observe that the proof of \cite[Lemma $3.1$]{GC-AC} is independent of the degree of polynomials $d$. 
Hence Lemma \ref{homology computation} also holds for all $d$. 
\end{rem}

We are now ready for the proof of Theorem $\ref{main-thm}$:
\begin{proof} 
Recall that $K[(I^e)_{ed+c}]$ is Koszul for all $c \geq \frac{2d}{3}$ and $e >0$ \cite[Corollary $6.10$]{CHTV}. 
We will show that $K[(I^e)_{ed+c}]$ is Koszul also for $\frac{d}{2} \leq c < \frac{2d}{3}$ and $e >0$. 

We set $B={S^{\prime}}/{(h_1,h_2)}$. By Lemma \ref{homology computation}$(i)$ and Proposition \ref{main-prop}, 
we may conclude that $B_{\bigtriangleup}$ is Koszul. One has also $\frac{B}{h_3B}=\Rees(I)$. It is enough to 
show for the $(c,e)$-diagonal 
subalgebra of Rees$(I)$, one has 
\begin{eqnarray*}
\reg_{B_{\bigtriangleup}}(\Rees(I)_{\bigtriangleup}) \leq 1. 
\end{eqnarray*}
Since $h_3t_3=0$ in $B$, we have a complex
\begin{eqnarray}\label{B-complex}
 {\bf{F:}}\; \cdots \xrightarrow{t_3\hspace*{.4cm}} B(-2d,-3) \xrightarrow{h_3\hspace*{.4cm}} B(-d,-2)
 \xrightarrow{t_3\hspace*{.4cm}} B(-d,-1) 
\xrightarrow{h_3\hspace*{.4cm}} B \longrightarrow 0,
\end{eqnarray}
where $F_0=B,\; F_{2i}=B(-id,-2i),\; F_{2i+1}=B(-(i+1)d,-2i-1)$. 
The homology of ${\bf{F}}$ can be described by using Lemma \ref{homology computation}
\begin{eqnarray*}
H_k({\bf{F}}) = \begin{cases}
    \text{Rees}(I)  &\text{ if  $ k=0 $,}\\
     0  &\text{ if  $k=2i$ and $i >0$,}\\ 
    [{S^{\prime}}/{(t_1,t_2,t_3)}](-(i+1)d-d,-2i-1)  &\text{ if $k=2i+1$ and $i \geq 0$.}
    \end{cases}  
\end{eqnarray*}
The assertion for $k=0$ holds by construction. For $k$ even and positive holds 
because of Lemma \ref{homology computation}.
For $k$ odd and positive by Lemma \ref{homology computation} $(ii)$, we have
\begin{eqnarray*}
H_{2i+1}({\bf{F}})= \frac{(t_3,f_3)}{(t_3,h_1,h_2)} (-(i+1)d,-2i-1).
\end{eqnarray*}
Hence $H_{2i+1}({\bf{F}})$ is cyclic generated by the residue class of 
$f_3 \mod (t_3,h_1,h_2)$ that has degree $(-(i+1)d-d,-2i-1)$. 
Using Lemma \ref{homology computation}$(iv)$ and keeping track of the 
degree we get the desired result. 
Applying $-_{\bigtriangleup}$ functor to $(\ref{B-complex})$, we obtain a complex ${ \bf F_{\bigtriangleup}: }$
\begin{eqnarray}\label{B-delta complex}
  \cdots  \xrightarrow{\hspace*{.5cm}} B(-2d,-3)_{\bigtriangleup}   
 \xrightarrow{\hspace*{.5cm}} B(-d,-2)_{\bigtriangleup} \xrightarrow{\hspace*{.5cm}}
 B(-d,-1)_{\bigtriangleup} \xrightarrow{\hspace*{.5cm}} B_{\bigtriangleup} \xrightarrow{\hspace*{.5cm}} 0,
\end{eqnarray}
where
\begin{eqnarray*}
 (F_{k})_{\bigtriangleup} = \begin{cases}
     B_{\bigtriangleup} &\text{ if $ k=0 $,}\\
     B(-id,-2i)_{\bigtriangleup}  &\text{ if $k=2i$,}\\ 
     B(-(i+1)d,-2i-1)_{\bigtriangleup} &\text{ if $k=2i+1$.}
    \end{cases} 
\end{eqnarray*}
Note that $H_{2i}({\bf{F_{\bigtriangleup}}})=0$ and $H_{0}({\bf{F_{\bigtriangleup}}})=\text{Rees}(I)_{\bigtriangleup}$. 
We observe that $H_{2i+1}({\bf{F_{\bigtriangleup}}})$ is not necessarily zero for all $e \geq 1$. Assume $e \geq 2$. 
Then we claim that $H_{2i+1}({\bf{F_{\bigtriangleup}}})=0$. Take the $j$-th degree component 
\begin{eqnarray*}
(H_{2i+1}({\bf{F}})_{\bigtriangleup})_j= [{S^{\prime}}/{(t_1,t_2,t_3)}]_{ (-(i+1)d-d+jc,-2i-1+je)}. 
\end{eqnarray*}
We will show that 
\begin{eqnarray}\label{extend-eqn}
[{S^{\prime}}/{(t_1,t_2,t_3)}]_{ (-(i+1)d-d+jc,-2i-1+je)}=0. 
\end{eqnarray}
Clearly $(\ref{extend-eqn})$ holds if $-2i-1+ej \geq 1$, that is, if $ej \geq 2(i+1)$. 
So, it is enough to show that $-(i+1)d-d+jc<0$ for $ej < 2(i+1)$, that is, $j< \frac{(i+2)d}{c}$ for 
$j < \frac{2(i+1)}{e}$. This is an easy consequence of the following inequalities:
\begin{eqnarray*}
  \frac{2(i+1)}{e} <  \frac{3(i+2)}{2} < \frac{(i+2)d}{c}.
\end{eqnarray*}
Assume $e=1$. We know that $H_{2i+1}({\bf{F_{\bigtriangleup}}})=H_{2i+1}({\bf{F}})_{\bigtriangleup}$. 
Take the $j$-th degree component of $H_{2i+1}({\bf{F}})_{\bigtriangleup}$, 
then $(H_{2i+1}({\bf{F}})_{\bigtriangleup})_j=0$ if $-(2i+1)+j \geq 1$, that is, if $j \geq 2i+2$. 
So, the largest degree of a non zero component of $H_{2i+1}({\bf{F_{\bigtriangleup}}})$ is at most $2i+1$. 
Therefore by $(\ref{koszul-test-2})$, 
one has $ \reg_{B_{\bigtriangleup}}  H_{k}(  {\bf{F_{\bigtriangleup}} )} \leq \reg_{S^{\prime}} H_{k}(  {\bf{F_{\bigtriangleup}} )}
\leq k$ for all $k \geq 1$. Applying Lemma \ref{regularity-complex-homology} to $(\ref{B-delta complex})$, we obtain
\begin{eqnarray}\label{alpha-beta}
 \reg_{B_{\bigtriangleup}}(\Rees(I)_{\bigtriangleup}) \leq 
\sup \{ \alpha^{\prime}, \beta^{\prime} \}, \text{ where}
\end{eqnarray}
\[
 \alpha^{\prime}=\sup \{ \reg_{B_{\bigtriangleup}}(F_{k})_{\bigtriangleup} -k:\;k\geq 0 \} \text{ and } 
\beta^{\prime}=\sup \{ \reg_{B_{\bigtriangleup}}   H_{k}(  {\bf{F_{\bigtriangleup}} )} -(k+1):\;k\geq 1 \}. 
\]
Since $B$ is defined by a regular sequence of elements of bidegree $(d,1)$, we may apply 
Proposition \ref{main-prop} to $(\ref{B-delta complex})$: 
\begin{eqnarray*}
 \reg_{B_{\bigtriangleup}}(F_k)_{\bigtriangleup} \leq 
\begin{cases}
   \max\{ \lceil \frac{id}{c} \rceil, \lceil \frac{2i}{e} \rceil  \}  &\text{ if  $ k=2i  $,}\\
   \max \{ \lceil \frac{(i+1)d}{c} \rceil,  \lceil \frac{2i+1}{e} \rceil\}  &\text{ if $ k=2i+1$.}
\end{cases}
\end{eqnarray*}
Since $ \frac{3}{2} < \frac{d}{c} \leq 2$, we conclude that $\alpha^{\prime} \leq 1$. Since 
$ \reg_{B_{\bigtriangleup}}  H_{k}(  {\bf{F_{\bigtriangleup}} )}  \leq k$ for all $k \geq 1$, we conclude that  
$\beta^{\prime} \leq -1 $. Therefore by $(\ref{alpha-beta})$, one has
\begin{eqnarray*}
 \reg_{B_{\bigtriangleup}}({\Rees(I)}_{\bigtriangleup}) \leq 1.  
\end{eqnarray*}
Thus we conclude that ${\Rees(I)}_{\bigtriangleup}$ is Koszul.
\end{proof}

\begin{rem}
We observe that in the proof of \cite[Theorem $3.2$]{GC-AC}, $H_{k}({\bf{F}})_{\bigtriangleup}=0$ for 
all $k$, whereas in our case, this is true for all $e\geq 2$, and not for $e=1$. This affects the proof 
of Theorem \ref{main-thm} very much from that of \cite[Theorem $3.2$]{GC-AC}. To achieve our goal 
we first have to deduce Lemma \ref{regularity-complex-homology}. We use the fact that if the homology 
module is non zero and its regularity is bounded by the homology module at zero position, then by 
$(\ref{koszul-test-2})$, Lemma \ref{regularity-complex-homology} and Proposition \ref{main-prop}, 
we conclude the proof.
\end{rem}

\section{More general base rings}
\medskip

In this Section, the two main results that we generalize are
\cite[Theorem $6.2$]{CHTV} and \cite[Corollary $6.10$]{CHTV}. 
We show that the Koszulness property holds even if the assumption of polynomial 
ring is replaced by a Koszul ring.  

Conca et al. in \cite{CHTV} posed two interesting questions at page $900$ 
one of which was positively answered by Aramova, Crona and De Negri \cite{ACN} 
who showed that for an arbitrary bigraded standard algebra $R$, the defining ideal of 
$R_{\bigtriangleup}$ has quadratic Grobner basis for $c,e \gg 0$, and 
another one by Blum \cite{Blum}, who showed that all the diagonal 
algebras of bigraded standard Koszul algebra $R$ are Koszul.
We will use these results in the proof of Theorem \ref{koszul-thm-gen} and 
Theorem \ref{poly to koszul change thm}.

Let $A$ and $B$ are two standard graded Koszul algebras. Let $A=K[A_1]$, where
$A_1=\langle X_1,\dots X_m \rangle$ is a $K$ vector space generated be linear forms 
with $\deg(X_i)=1$. Similarly let $B=K[B_1]$, where $B_1=\langle Y_1,\dots Y_n \rangle$ 
is a $K$ vector space generated be linear forms with $\deg(Y_j)=1$. We set $T=A \tensor_K B$. 
Then $T$ is bigraded standard by setting $\deg(X_i)=(1,0)$ and $\deg(Y_j)=(0,1)$. Let $R$ be a 
bigraded quotient of $T$, that is, $R=T/I$ for some bihomogeneous ideal $I$ of $T$. 

Let $c$ and $e$ be positive integers. We will study the Koszul property of $(c,e)$-diagonal subalgebra 
$R_{\bigtriangleup}$ of bigraded algebra $R$. Consider the bigraded free resolution of $R$ over $T$:
\begin{eqnarray}\label{free res 1}
  \cdots \longrightarrow F_i \longrightarrow \cdots \longrightarrow F_1 \longrightarrow T 
  \longrightarrow R \longrightarrow 0, 
\end{eqnarray}
where 
\begin{eqnarray*}
F_i= \bigoplus_{(a,b) \in \mathbb{N}^2}T(-a,-b)^{\beta_{i,(a,b)}}. 
\end{eqnarray*}
Set 
\begin{eqnarray}\label{ext-eqn3}
 t_{i,1}=\max \{a : \; \exists \;b \;\text{  s.t. } \; \beta_{i,(a,b)} \neq 0 \}, \text{ and  } 
 t_{i,2}=\max \{b : \; \exists \;a \;\text{  s.t. } \; \beta_{i,(a,b)} \neq 0 \}.
\end{eqnarray}
With this notation, the following is a generalization of \cite[Theorem $6.2$]{CHTV}:
\begin{thm}\label{koszul-thm-gen} 
Let $A$ and $B$ are standard graded Koszul algebras. We set $T=A \tensor_K B$ and 
$R$ a bi-graded quotient of $T$. Then:
\begin{itemize}
 \item[(i)] $T$ is Koszul and $\reg_{T} R $ is finite.
 \item[(ii)] If $c \geq  \sup \{ \frac{t_{i,1}}{i+1}:\;  i \geq 1 \} \in \mathbb{R}$ and 
 $e \geq  \sup \{ \frac{t_{i,2}}{i+1}:\;  i \geq 1 \} \in \mathbb{R} $, then $R_{\bigtriangleup}$ is Koszul.
 \item[(iii)] In particular, if $ c \geq \frac{\reg_{T} R -1}{2}$ and $ e \geq \frac{\reg_{T} R -1}{2}$, 
 then $R_{\bigtriangleup}$ is Koszul. 
\end{itemize}
\end{thm}

\begin{proof}
For (i) $T$ is Koszul by \cite{BF} and $\reg_{T} R $ is finite by \cite[Theorem $1$]{AE}. For (ii) 
applying $-_{\bigtriangleup}$ functor to bigraded free resolution $(\ref{free res 1})$ of 
$R$ over $T$, one obtains an exact complex
\begin{eqnarray}\label{rees free res-delta}
  \cdots \longrightarrow (F_i)_{\bigtriangleup} \longrightarrow \cdots \longrightarrow (F_1)_{\bigtriangleup} 
  \longrightarrow T_{\bigtriangleup} \longrightarrow R_{\bigtriangleup} \longrightarrow 0, 
\end{eqnarray}
 of $T_{\bigtriangleup}$-modules, where 
\begin{eqnarray*}
 (F_i)_{\bigtriangleup}= \bigoplus_{(a,b) \in \mathbb{N}^2}T(-a,-b)_{\bigtriangleup}^{\beta_{i,(a,b)}}.
\end{eqnarray*}
Note that $T_{\bigtriangleup}= A^{(c)} \underline{\otimes} B^{(e)}$, where $A^{(c)}$ denotes the c-th Veronese 
subring of $A$, and $B^{(e)}$ denotes the e-th Veronese subring of $B$. Note that $T_{\bigtriangleup}$ is Koszul 
\cite[Theorem $2.1$]{Blum}. To show that $R_{\bigtriangleup}$ is Koszul,  it is enough to show that
\[
   \reg_{ T_{\bigtriangleup}}{R_{\bigtriangleup}} \leq 1.
\]
Applying \cite[Lemma $6.3$(ii)]{CHTV} to $(\ref{rees free res-delta})$, we get 
\begin{eqnarray}
 \reg_{ T_{\bigtriangleup}}{R_{\bigtriangleup}} \leq \sup 
\{  \reg_{ T_{\bigtriangleup}}{({F_i})}_{\bigtriangleup}-i \; :\; i \geq 1 \}. 
\end{eqnarray}
Thus, it is enough to show that
\begin{eqnarray}\label{ext-eqn5}
 \reg_{ T_{\bigtriangleup}}{({F_i})}_{\bigtriangleup} -i \leq 1 \text{ for all $i \geq 1$.}
\end{eqnarray}
Since 
\[
 (F_i)_{\bigtriangleup}= \bigoplus_{(a,b) \in \mathbb{N}^2}T(-a,-b)_{\bigtriangleup}^{\beta_{i,(a,b)}},
\]
one has 
\begin{eqnarray}\label{main-eqn-1-gen}
  \reg_{ T_{\bigtriangleup}}{({F_i})}_{\bigtriangleup} =  
  \max\{ \reg_{T_{\bigtriangleup}}T(-a,-b)_{\bigtriangleup} \;:\;\beta_{i,(a,b)} \neq 0 \}.
\end{eqnarray}
Now we need to evaluate $\reg_{T_{\bigtriangleup}}T(-a,-b)_{\bigtriangleup} $. 
We denote by $V_A(c,\alpha)$, the Veronese modules of $A$, that is, 
$V_A(c,\alpha) =\bigoplus_{s \in \mathbb{N}}A_{sc+\alpha}$ for $\alpha=0,\dots,c-1$. 
Similarly denote $V_B(e,\beta)$, the Veronese modules of $B$. 

Hence for the shifted module $T(-a,-b)_{\bigtriangleup}$, we can write
\begin{eqnarray*}\label{ext-eqn4}
 T(-a,-b)_{\bigtriangleup} =\; \bigoplus_{s} \; [ A_{sc-a}\; \tensor \; B_{se-b} ] \;=\;
 V_A(c,\alpha) (-\lceil \dfrac{a}{c}  \rceil)\; \underline{\otimes} \;
 V_B(e,\beta) (-\lceil \dfrac{b}{e}  \rceil),
\end{eqnarray*}
where $\alpha=-a \mod (c),\; 0 \leq \alpha \leq c-1$ and
$\beta=-b \mod (e),\; 0 \leq \beta \leq e-1$. 

The Veronese modules $V_A(c,\alpha)$ and $V_B(e,\beta)$ have a linear resolutions as a 
$A^{(c)}$-module and $B^{(c)}$-module respectively, see \cite[Lemma $5.1$]{BCR}. Hence by 
$(\ref{regularity relation segre product})$, one has 
\begin{eqnarray}\label{regularity relation segre product-main-thm}
\reg_{T_{\bigtriangleup}}   {T(-a,-b)_{\bigtriangleup}}= 
\max \{  \lceil \dfrac{a}{c}  \rceil, 
\lceil \dfrac{b}{e}  \rceil   \}. 
\end{eqnarray}
Thus by $(\ref{ext-eqn3})$, $(\ref{ext-eqn5})$, $(\ref{main-eqn-1-gen})$ and 
$(\ref{regularity relation segre product-main-thm})$, we conclude that $R_{\bigtriangleup}$ is Koszul 
provided that 
\begin{eqnarray}\label{regularity relation-main-thm}
 \max \{ \lceil \dfrac{t_{i,1}}{c}  \rceil, \;
\lceil \dfrac{t_{i,2}}{e}  \rceil  \} \leq i+1 \text{ for all $i\geq 1$.}
\end{eqnarray}
From $(\ref{regularity relation-main-thm})$, we conclude that if 
$c \geq  \sup \{ \frac{t_{i,1}}{i+1}:\;  i \geq 1 \} $ and 
$e \geq  \sup \{ \frac{t_{i,2}}{i+1}:\;  i \geq 1 \} $, then $R_{\bigtriangleup}$ is Koszul. 
By definition, one has $ t_{i,1} \leq t_i \leq \reg_T R -i$. Thus we have:
\begin{eqnarray}\label{regularity relation-main-thm-11}
  \dfrac{t_{i,1}}{i+1} \leq   \dfrac{t_{i}}{i+1} \leq   \dfrac{\reg_T R -i}{i+1}. 
\end{eqnarray}
We know that $\reg_T R $ is finite. Notice that $\dfrac{\reg_T R -i}{i+1}$ is a 
decreasing function of $i$, as $i$ vary over the natural numbers. Taking sup in $(\ref{regularity relation-main-thm-11})$, we get
\[
 \sup \{ \frac{t_{i,1}}{i+1}:\;  i \geq 1 \} \leq  \sup \{ \dfrac{t_{i}}{i+1} :\;  i \geq 1  \}  \leq  
 \sup \{ \dfrac{\reg_T R -i}{i+1} :\;  i \geq 1 \}. 
\]
Similarly the other case
\[
 \sup \{ \frac{t_{i,2}}{i+1}:\;  i \geq 1 \} \leq  \sup \{ \dfrac{t_{i}}{i+1} :\;  i \geq 1  \}  \leq  
 \sup \{ \dfrac{\reg_T R -i}{i+1}  :\;  i \geq 1 \}.
\]
Note that 
\[
 \sup \{ \dfrac{\reg_T R -i}{i+1} :\;  i \geq 1 \} \leq \dfrac{\reg_T R -1}{2}.
\]
Thus we observe that the numbers $\sup \{ \frac{t_{i,1}}{i+1}:\;  i \geq 1 \}$ and 
$\sup \{ \frac{t_{i,2}}{i+1} :\;  i \geq 1 \}$ are in fact some finite real numbers, bounded by 
$\dfrac{\reg_T R -1}{2}$. Thus the claim (ii) and (iii) follows.

\end{proof}

\begin{rem}
 Note that in the Theorem \ref{koszul-thm-gen}, if we take $A=K[x_1,\cdots,x_m]$ and 
 $B=K[y_1,\cdots,y_n]$, then we will re-obtain \cite[Theorem $6.2$]{CHTV}.
\end{rem}

The following is the generalization of \cite[Corollary $6.9$]{CHTV}:

\begin{cor}\label{koszul cor}
Let $I$ be a homogeneous ideal of a standard graded Koszul ring $A$. 
Let $d$ denotes the highest degree of a minimal generator of $I$. 
Then there exists integers $c_0,e_0$ such that the $K$-algebra $K[(I^e)_{ed+c}]$ 
is Koszul for all $c \geq c_0$ and $e \geq e_0$.
\end{cor}
\begin{proof}
Let $B=K[t_1,\dots,t_k]$ in Theorem \ref{koszul-thm-gen}. Then $T=A[t_1,\dots,t_k]$ is 
a polynomial extension of $A$. Note that $T$ is a bigraded standard algebra by 
setting $\deg X_i=(1,0)$ and $\deg t_j=(0,1)$. By replacing $I$ with the 
ideal generated by $I_d$, we may assume that $I$ is generated by forms of degree $d$. 
Then Rees$(I) \subset A[t]$ is a bigraded standard algebra by setting $\deg X_i=(1,0)$ 
and $\deg ft=(0,1)$ for all $f \in I_d$. Moreover Rees$(I)$ can also be realized as 
the bigraded quotient of $T$. 

By Theorem \ref{koszul-thm-gen}$(i)$, we have $T$ is Koszul and $\reg_{T} \text{Rees$(I)$} $ is finite. 
Note that the numbers $c$ and $e$ exist from Theorem \ref{koszul-thm-gen}$(ii)$ such 
that Rees$(I)_{\bigtriangleup}$ is Koszul. In particular, if $c,e \geq \frac{\reg_{T} 
\text{Rees$(I)$} -1}{2}$, then Rees$(I)_{\bigtriangleup}$ is Koszul by 
Theorem \ref{koszul-thm-gen}$(iii)$. Thus the claim follows, 
since Rees$(I)_{\bigtriangleup}=K[(I^e)_{ed+c}]$. 
\end{proof}
 
\begin{rem}
In the Corollary \ref{koszul cor}, the integers $c_0$ and $e_0$ can be computed explicitly whenever 
 one knows the shifts in the bigraded free resolution of Rees$(I)$ over the Koszul ring $T$. 
 For instance, when $I$ is a complete intersection ideal generated by homogeneous forms of degree $d$, 
 one has Theorem \ref{poly to koszul change thm}.
\end{rem}
The following is the generalization of \cite[Corollary $6.10$]{CHTV}:

\begin{thm}\label{poly to koszul change thm}
Let $A$ be a standard graded Koszul ring. Let $I$ be an ideal of $A$ generated by a 
regular sequence $f_1,f_2, \dots,f_k$, of homogeneous forms of degree $d$. 
Then $K[(I^e)_{ed+c}]$ is Koszul for all $c \geq \frac{d(k-1)}{k}$ and $e >0$.
\end{thm}
\begin{proof}
Let $A=K[A_1]$, where $A_1=\langle X_1,\dots X_m \rangle$ is a $K$ vector space generated be linear forms 
with $\deg(X_i)=1$. Let $A^{\prime}=A[t_1,\dots,t_k]$ be a polynomial extension of $A$. 
Then $A^{\prime}$ is a bigraded standard algebra by setting $\deg X_i=(1,0)$ and 
$\deg t_j=(0,1)$. Let $I$ be an ideal of $A$ generated by a 
regular sequence $f_1,f_2, \dots,f_k$, of homogeneous forms of degree $d$. Then 
Rees$(I)  \subset A[t]$ is a bigraded standard algebra by setting $\deg X_i=(1,0)$ and 
$\deg f_jt=(0,1)$. Let 
\[
 \phi : A[t_1,t_2,\dots,t_k] \longmapsto A[It]
\]
be the surjective map by sending $t_j$ to $f_jt$. Since $I$ is a complete intersection ideal, one has
\[
 \ker(\phi)= I_2
\begin{pmatrix}
 f_1 & f_2 & \dots & f_k \\
t_1 & t_2 & \dots & t_k
\end{pmatrix}.
\]
The resolution of $A[It]$ over $A^{\prime}$ is given by the Eagon-Northcott complex $(\ref{ENcomplex})$. 
We know that $\ker(\phi)$ is a determinantal ideal and $\grade(\ker(\phi))=k-1$, hence the Eagon-Northcott 
complex $(\ref{ENcomplex})$ is the minimal free resolution of $A[It]$ over $A^{\prime}$:
\begin{eqnarray}\label{ENcomplex}
 0 \longrightarrow F_{k-1} \longrightarrow \cdots \longrightarrow F_1 \longrightarrow 
F_0 \longrightarrow A[It] \longrightarrow 0, 
\end{eqnarray}
where 
\[
 F_i= \bigoplus_{j=1}^{i} A^{\prime}(-jd,-i-1+j)^{\sharp_i}.
\]
Here $\sharp_i$ denotes some integer which is irrelevant in our discussion. 
Applying $-_{\bigtriangleup}$ functor to $(\ref{ENcomplex})$, one obtains an exact complex:
\begin{eqnarray}\label{delta-ENcomplex}
 0 \longrightarrow (F_{k-1})_{\bigtriangleup} \longrightarrow \cdots \longrightarrow (F_1)_{\bigtriangleup} 
\longrightarrow (F_0)_{\bigtriangleup} \longrightarrow A[It]_{\bigtriangleup} \longrightarrow 0
\end{eqnarray}
of $A^{\prime}_{\bigtriangleup}$-modules, where
\[
(F_i)_{\bigtriangleup}= \bigoplus_{j=1}^{i} A^{\prime}(-jd,-i-1+j)_{\bigtriangleup}^{\sharp_i}. 
\]
From the exact complex  $(\ref{delta-ENcomplex})$, we build another complex: 
\[ 
{\bf{F}} :\;\; 0 \longrightarrow (F_{k-1})_{\bigtriangleup} \longrightarrow \cdots \longrightarrow (F_1)_{\bigtriangleup} 
\longrightarrow (F_0)_{\bigtriangleup} \longrightarrow  0
\]
of $A^{\prime}_{\bigtriangleup}$-modules. Then the homology of the new complex ${\bf{F}}$ is given by
\[
 H_i({\bf{F}}) = 
\begin{cases}
   A[It]_{\bigtriangleup} ,  &\text{ if $i=0$;}\\
    0,  &\text{ if $ i> 0$.}
\end{cases}
\]
Thus Lemma \ref{regularity-complex-homology} 
applied to the complex ${\bf{F}}$, one has
\[
 \reg_{ A^{\prime}_{\bigtriangleup}}{(A[It]_{\bigtriangleup})} \leq \sup 
\{  \reg_{ A^{\prime}_{\bigtriangleup}}{({F_i})}_{\bigtriangleup}-i \;:\; i=1,2,\dots,{k-1} \}. 
\]
Note that $A_{\bigtriangleup}^{\prime}$ is Koszul \cite[Theorem $2.1$]{Blum}. 
To show that $A[It]_{\bigtriangleup}$ is Koszul, it is enough to show 
\begin{eqnarray}\label{main-eqn}
\reg_{ A^{\prime}_{\bigtriangleup}}{({F_i})}_{\bigtriangleup}-i \leq 1 
\text{ for all $i=1,2,\dots,{k-1}$.} 
\end{eqnarray}
One has
\begin{eqnarray*}
 \reg_{ A^{\prime}_{\bigtriangleup}}{({F_i})_{\bigtriangleup}}=
\max\{ \reg_{A_{\bigtriangleup}^{\prime}}A^{\prime}(-jd,-i-1+j)_{\bigtriangleup} \; :\; j=1,2,\dots,i \}.
\end{eqnarray*}
By similar argument as used in Theorem \ref{koszul-thm-gen} to obtain 
Equation $(\ref{regularity relation segre product-main-thm})$, we get 
\begin{eqnarray*}
 \reg_{A_{\bigtriangleup}^{\prime}}{({F_i})_{\bigtriangleup}}= \max \{\; \lceil \dfrac{jd}{c}  \rceil, \;
\lceil \dfrac{i+1-j}{e}  \rceil \;:\; j=1,2,\dots,i \} . 
\end{eqnarray*}
Thus we have
\begin{eqnarray}\label{main-eqn-2}
\reg_{A_{\bigtriangleup}^{\prime}}{({F_i})_{\bigtriangleup}}= \max \{\; \lceil \dfrac{id}{c}  \rceil,\; 
\lceil \dfrac{i}{e}  \rceil \; \} . 
\end{eqnarray}
Therefore by (\ref{main-eqn}) and (\ref{main-eqn-2}), we conclude that $K[(I^e)_{ed+c}]$ is Koszul 
if $c \geq \frac{d(k-1)}{k}$ and $e >0$.
\end{proof}

We conclude the section with one final remark, and with an open question: 
\begin{rem}
The claim by Conca, Herzog, Trung and Valla in \cite[p. 900]{CHTV} together with 
Theorem \ref{main-thm} and Theorem \ref{poly to koszul change thm} in this paper, suggest that the 
following question may have a positive answer.
\end{rem}

\begin{quest}
Let $I$ be an ideal of a Koszul ring $A$ generated by a regular sequence 
 $f_1,f_2, \dots,f_k$, of homogeneous forms of degree $d$. Is it true that 
 $K[(I^e)_{ed+c}]$ is Koszul for all $c \geq \frac{d}{2}$ and $e >0$.
\end{quest}

\end{document}